\RequirePackage{fix-cm}
\documentclass[smallextended]{svjour3}       
\smartqed  
\usepackage{graphicx}
\usepackage{amsfonts}
\usepackage{mathrsfs}
\usepackage{amsmath}
\begin{document}

\title{Quotients of affine connection control systems 
}


\author{Qianqian Xia and Zhiyong Geng         
}


\institute{Qianqian  Xia\\ Department of Mechanics and Aerospace Engineering, Peking University, Beijing 100871, China\\
              \email{inertialtec@sina.com}
}

\date{Received: date / Accepted: date}

\maketitle

\begin{abstract}
In this paper, we investigate the existence of a subclass of quotients of affine connection control systems, which preserve the mechanical structures.
Both local and global sufficient and necessary conditions are given for the geodesically accessible affine connection control systems such that they can admit this subclass of quotients. The structural properties of the quotient map and the quotient mechanical control system are discussed. 
\end{abstract}
section{Introduction}

Quotients of affine control systems have been studied by several authors.
Due to the complicated nature of the original affine control systems, it is
desirable to study their properties by investigating their quotient affine control
systems.
Given an affine control system
\begin{equation}\label{Equ1}
\dot y = {f_0}(y) + \sum\limits_{i = 1}^r {{u_i}{f_i}(y)}, y \in M,
u_i \in {\mathbb{R}}, i= 1, \cdots ,r,
\end{equation}
where $M$ is a smooth n-dimensional manifold, a curve $y:I \to M$ is
a trajectory of the system ($\ref{Equ1}$) if it satisfies the above
equation. An affine control system
\begin{equation}\label{Equ2}
\dot z = {h_0}(z) + \sum\limits_{j = 1}^s {{v_j}{h_j}(z)}, z \in N,
v_j \in {\mathbb{R}}, j= 1, \cdots ,s
\end{equation}
on a smooth m-dimensional manifold $N$ is called the quotient system
of the system ($\ref{Equ1}$) if there exists a smooth surjective map
$\psi:M \to N$, such that whenever $y(t), t \in I$ is a trajectory
of the system ($\ref{Equ1}$), then $z(t) = \psi (y(t)),t \in I$ is a
trajectory of the system ($\ref{Equ2}$). The map $\psi$ is then
called a quotient map.

Some subclasses of quotients of affine control systems, such as quotients induced by symmetries \cite{1,2} and controlled invariance \cite{3}, have been studied by several authors. The quotients of some subclasses of affine control systems, such as Hamiltonian control system \cite{4} and hybrid control systems \cite{5} have also been investigated.

In this paper, we study the quotients of affine connection
control systems which can be considered as a subclass of affine control systems.

Based on the Lagrangian point of view on mechanics, a mechanical
control system \cite{6} can be given by  $4$-tuple $(Q,\nabla ,\Im
_0 ,d)$, where

(1) $Q$ is an $n$-dimensional configuration manifold;

(2) $\nabla$ is a symmetric affine connection on $Q$;

(3)$\Im _0  = (g_0 ,g_1 , \cdots ,g_m )$ is an $(m+1)$-tuple of smooth
vector fields on $Q$;

(4) $d:TQ \to TQ$ is a map sending the fiber $T_q Q$
  into the fiber $T_q Q$, for any $q \in Q$, linear on fibers.\\
Here $d$ represents the dissipative-type or gyroscopic-type forces
acting on the system, $g_0$ represents an uncontrolled force which
can be potential and $g_1,\cdots, g_m$ represent controlled forces.

A curve $\gamma :I \to Q,I \subseteq R$ is a trajectory of a
mechanical control system if it satisfies the equation:
\begin{equation}\label{Equ3}
\nabla _{\dot \gamma (t)} \dot \gamma (t) = g_0 (\gamma (t)) +
d(\dot \gamma (t)) + \sum\limits_{i = 1}^m {u_r g_r } (\gamma (t)).
\end{equation}
As a system on $TQ$, we have:
\begin{equation}\label{Equ4}
\dot \Upsilon (t) = S(\Upsilon (t)) + g_0 ^{vlft} (\gamma (t)) +
d^{vlft} (\dot \gamma (t)) + \sum\limits_{r = 1}^m {u_r g_r ^{vlft}
} (\gamma (t)).
\end{equation}

An affine connection control system $(Q,\nabla ,\Im )$ is a
mechanical control system for which $ d = 0,g_0  = 0$:
\begin{equation}\label{Equ5}
\nabla _{\dot \gamma (t)} \dot \gamma (t) = \sum\limits_{r = 1}^m
{u_r g_r } (\gamma (t)).
\end{equation}
The above equation can be written as:
\begin{equation}\label{Equ6}
\dot \Upsilon (t) = S(\Upsilon (t)) + \sum\limits_{r = 1}^m {u_r g_r
^{vlft} } (\gamma (t)).
\end{equation}

The investigation of morphisms in the category of affine connection
control systems was initiated by Lewis in \cite{7}, where the maps
there were assumed to be of the special form $(\phi, D\phi )$, with
$\phi$ a map between the configuration spaces. The state equivalence
problems for mechanical control systems were first studied clearly
in \cite{8}, where no particular forms of the diffeomorphisms
between the tangent bundles were assumed. A subclass of mechanical
control systems, namely that of the geodesically accessible systems were
introduced, which have been proved to possess unique mechanical
structures. In this paper, we will study the quotient mechanical
structures for geodesically accessible systems. We are particularly interested
in the quotient maps retaining the mechanical structures. Precisely,
given a geodesically accessible affine connection control system $(Q,\nabla ,\Im )$ , we will be interested in
answering the following two questions:

(i) Given a point $(x_0, v_0) \in TQ$ with $ v_0 \in T_{x_0}Q$, when does there exist a mechanical control system $(\tilde
Q,\tilde \nabla ,\tilde \Im _0 ,d)$ and a quotient map $\tau : U \to V$, such that
\begin{eqnarray}\label{Equ7}
 &&\tau _ *  S = \tilde S + \tilde g_0^{vlft}  + \tilde d^{vlft},\nonumber \\
 &&\tau _ *  g_r^{vlft}  = \tilde g_r^{vlft}?
\end{eqnarray}
Here we denote by $U$ and $V$  neighborhoods, respectively of
$(x_0, v_0) \in TQ$ and $(\tilde {x}_0, \tilde {v}_0) \in T\tilde
Q$.

(ii) When does there exist a mechanical control system $(\tilde
Q,\tilde \nabla ,\tilde \Im _0 ,d)$ and a quotient map $\tau : TQ \to T\tilde Q$ satisfying ($\ref{Equ7}$)?

We say that the affine connection control system $(Q,\nabla ,\Im )$
admits locally (globally) a quotient mechanical system if question (i) ((ii)) is solvable.

Comparing with the problem given in \cite{7}, no particular forms of
maps between the tangent bundles are assumed. It is our main
concern to recognize the structure of the quotient map
that gives a quotient mechanical control system and make a clear study
of the existence of quotients of mechanical control systems which are
also mechanical.

The main results of the paper are given in Theorem $\ref{result1}$ and
Theorem $\ref{result8}$, which are the solutions,  respectively of
question (i) and (ii). For geodesically accessible affine connection
control systems  $(Q,\nabla ,\Im )$ , we prove that if question (i)
is solvable, the quotient map must have the form $(\phi, D\phi )$,
where $\phi$ is a submersion between the configuration spaces.
Besides, the quotient mechanical system $(\tilde Q,\tilde \nabla
,\tilde \Im _0 ,d)$ is proved to be a geodesically accessible affine
connection system. Under the assumption of the completeness of some
vector fields, we show the fiber bundle structure of the
configuration manifold $Q$ such that the affine connection control
system $(Q,\nabla ,\Im )$ can admit globally a quotient mechanical
system. For detailed explanations for the notations involved above,
see section~\ref{sec2}

The paper is organized as follows. Preliminaries are provided in
section \ref{sec2}  Question (i) is investigated in
section~\ref{sec3} Question (ii) is studied in section \ref{sec4}
Conclusions are given in section \ref{sec5}

Throughout the paper, "smooth"  means analytic. Some of our results hold in the
$C^\infty$ category, but for all we say to be true, we need analyticity, so we make this a blanket assumption.
\section{Preliminaries}\label{sec2}

Let $Q$ be an $n$-dimensional manifold.  We denote by $C^\infty  (Q)$ the set of $C^\infty$ functions on $Q$ and by  $\Gamma^\infty (TQ)$ the set of  sections of the tangent bundle $TQ$.
An affine connection on $Q$ is a
map\[\nabla :\Gamma^\infty (TQ) \times \Gamma^\infty (TQ) \to \Gamma^\infty (TQ)\] \[(X,Y)
\to \nabla _X Y\] which satisfies the following conditions:

(AC1)$(X,Y) \to \nabla _X Y$ is $\mathbb{R}$-bilinear;

(AC2)$\nabla _{fX} Y = f\nabla _X Y$ for $f \in C^\infty  (Q)$;

(AC3)$\nabla _X (fY) = f\nabla _X Y + (L_X f)Y$ for $f \in
C^\infty(Q)$.

In local coordinates $(x^1  \cdots ,x^n )$ on $Q$, the Christoffel
symbols $\Gamma _{ij}^k $ for an affine connection $\nabla$ are
defined by \[\nabla _{\frac{\partial }{{\partial x^i }}}
\frac{\partial }{{\partial x^j }} = \Gamma _{ij}^k \frac{\partial
}{{\partial x^k }}, i, j = 1, \cdots ,n.\]

A geodesic of an affine connection $\nabla $ on $Q$ is a smooth
curve $\gamma $ on $Q$ satisfying $\nabla _{\dot \gamma (t)} \dot
\gamma (t) = 0$. In local coordinates, the equations for a geodesic
have the form\[ \ddot x^i  + \Gamma _{jk}^i \dot x^j \dot x^k  = 0,
i = 1, \cdots, n.\] These second-order equations are equivalent to
the system of first-order equations on $TQ$ given by
\begin{eqnarray}
 &&\dot x^i  = y^i, i = 1, \cdots, n,\nonumber \\
 &&\dot y^j  =  - \Gamma _{kl}^j y^k y^l, j = 1, \cdots, n. \nonumber
\end{eqnarray}
These equations define a vector field $S$ on $TQ$ which is called
 the geodesic spray of $\nabla $ and is given in local coordinates by
 \[S = y^i \frac{\partial }{{\partial x^i }} - \Gamma _{kl}^j y^k y^l
\frac{\partial }{{\partial y^j }}.\]

Let $D$ be a distribution on a manifold $Q$ with an affine connection
$\nabla$. we shall denote by $\Gamma^\infty (D)$  the set of sections
of the distribution $D$. The distribution $D$ is called geodesically invariant if
for every geodesic $c:[a,b] \to Q$ of $\nabla $, $\dot c(a) \in
D_{c(a)} $ implies that $\dot c(t) \in D_{c(t)}$ for each $t \in [a,b]$. We say $D$ is totally geodesic if it is
involutive and geodesically invariant. The affine connection $\nabla$ restricts to $D$ if
${\nabla _X}Y \in \Gamma^\infty (D)$ for any $X \in \Gamma^\infty (TQ)$ and $Y \in \Gamma^\infty (D)$.

The symmetric product of $X,Y \in \Gamma^\infty (TQ)$ is the vector field
\[\left\langle {X:Y} \right\rangle  = \nabla _X Y + \nabla _Y X.\] We use
$Sym(D)$ to denote the smallest distribution on $Q$ containing $D$ and
such that it is closed under the symmetric product. A distribution $D$
on $Q$ is geodesically invariant if and only if $\left\langle {X:Y}
\right\rangle \in \Gamma^\infty (D)$ for every $X,Y \in \Gamma^\infty (D)$.

We define the torsion tensor $T:\Gamma^\infty (TQ) \times \Gamma^\infty (TQ) \to \Gamma^\infty (TQ)$,
$T(X,Y) = \nabla _X Y - \nabla _Y X - [X,Y]$. We say $\nabla $ is
symmetric if $T(X,Y) = 0$ for all $X,Y \in \Gamma^\infty (TQ)$. In local
coordinate, it gives $\Gamma _{ij}^k  = \Gamma _{ji}^k ,1 \le i,j,k
\le n$. We define the curvature tensor $R:\Gamma^\infty (TQ) \times \Gamma^\infty (TQ) \times \Gamma^\infty (TQ)\to \Gamma^\infty (TQ)$,
$R(X,Y)W = \nabla _X \nabla _Y W - \nabla _Y \nabla _ X W - {\nabla _{[X,Y]}}W$. Throughout the paper, all affine connections are assumed to be symmetric.

For the affine connection $\nabla $ on $Q$, we have the splitting  ${T_{{v_q}}}TQ \simeq {H_{{v_q}}}(TQ) \oplus {V_{{v_q}}}(TQ)$ for
any ${v_q} \in TQ$. Here $HTQ$ denotes the horizontal subbundle and $VTQ$ denotes the vertical subbundle.
At each ${v_q} \in TQ$, the linear map ${T_{{v_q}}}{\tau _Q}:{T_{{v_q}}}TQ \to {T_q}Q$, where ${\tau _Q}:TQ \to Q$ denotes the projection map,
restricted to the horizontal subspace ${H_{{v_q}}}TQ$, is an isomorphism. We define the horizontal lift of $u_q$ to $v_q$ as:
\[hlft({v_q},{u_q}) = {({T_{{v_q}}}{\tau _Q}\left| {{H_{{v_q}}}} \right.TQ)^{ - 1}}({u_q}).\]
Then the horizontal lift of the vector field $X \in \Gamma^\infty (TQ)$ is the vector field ${X^H} \in \Gamma^\infty (TTQ)$ defined by ${X^H}({v_q}) = hlft({v_q},X(q))$, where $\Gamma^\infty (TTQ)$ denotes the set of sections of $TTQ$.
The vertical lift of $X \in \Gamma^\infty (TQ)$ is defined as \[X^{vlft} (v_q) = \frac{d}{{dt}}(v_q +
tX(q))\left| {_{t = 0} } \right., v_q  \in TQ.\]

In local coordinates, if $X = {X^i}\frac{\partial }{{\partial q{}^i}}$, then we have
\begin{equation}\label{Equ8}
{X^H}({v_q}) = {X^i}(q)(\frac{\partial }{{\partial {q^i}}} - \Gamma _{ik}^j(q){v^k}\frac{\partial }{{\partial {v^j}}}),
\end{equation}
\begin{equation}\label{Equ9}
X^{vlft}(v_q ) = X^i(q)\frac{\partial }{{\partial v^i }}.
\end{equation}

We have
\begin{equation}\label{Equ10}
[S,{X^{vlft}}]({v_q}) =  - X \oplus {\nabla _{{v_q}}}X,
\end{equation}
\begin{equation}\label{Equ11}
[X_1^{vlft},[S,X_2^{vlft}]]({v_q}) = 0 \oplus \left\langle {{X_1}:{X_2}} \right\rangle,
\end{equation}
\begin{equation}\label{Equ12}
[S,[S,{X^{vlft}}]]({v_q}) =  - 2{\nabla _{{v_q}}}X \oplus (R(X,{v_q}){v_q} + ver(\nabla _{{v_q}^H}^H{({\nabla _{{v_q}}}X)^{vlft}})),
\end{equation}
where ${\nabla ^H}$ denotes the horizontal lift of a connection and
$ver$ denotes the vertical projection map
$ve{r_{{v_q}}}:{T_{{v_q}}}TQ \to {V_{{v_q}}}{TQ}$.
\section{Quotient mechanical control systems induced by invariant
distributions}\label{sec3}
In this section we investigate question (i) given in the introduction.
\begin{definition}{Definition} \cite{8}
The system ($\ref{Equ3}$) is called geodesically
accessible at $x_0  \in Q$ if $Sym(g_1 , \cdots ,g_m )(x_0 ) =
T_{x_0 } Q$ and geodesically accessible if the above equality holds
for all $x_0  \in Q$.
\end{definition}
Our solutions of question (i) is the central result of this section and for the convenience of the reader we will formulate it now.
\begin{theorem}{Theorem}\label{result1}
Question (i) is solvable if and only if
there exists an involutive distribution $D$  with the affine connection $\nabla$ restricting to $D$, and the curvature tensor satisfying $R(X,v)v \in \Gamma^\infty  (D)$, for any $X \in \Gamma^\infty (D)$, $v \in \Gamma^\infty (TQ)$, besides,  $[g_i ,D] \subseteq D,i = 1, \cdots ,m$.
\end{theorem}
Recall that a Riemannian manifold is said to be irreducible if  the representation of the holonomy group in $T_xQ$
is irreducible (see \cite{9}). Then a corollary follows immediately.
\begin{corollary}{Corollary}
Let $(Q,g)$ be an irreducible Riemannian manifold with $\nabla$ the
corresponding Levi-Civita connection. Then the question (i) is not solvable for any given geodesically
accessible affine connection control system $(Q,\nabla ,\Im )$.
\end{corollary}
\begin{proof}
We know that a Riemannian manifold is irreducible if and only if the
only totally geodesic distribution on $Q$ is $TQ$ (see \cite{9}). Then the results follow immediately from Theorem \ref{result1}.
\end{proof}
Recall that a smooth involutive distribution $D$ is called invariant for the affine control system (\ref{Equ1})
if $[f_0,D]\subseteq D$ and $[f_i ,D]\subseteq D,i = 1,\cdots,r$  (see \cite{3}). In a neighborhood $V$ of $y_0$ with local coordinates $(y^1,y^2)$,
where $D = span\left\{ {\frac{\partial }{{\partial {y^1}}}} \right\}$, we have the following
local decomposition of the system (\ref{Equ1}):
\begin{eqnarray}
 &&{{\dot y}^1} = f_0^1({y^1},{y^2}) + \sum\limits_{i = 1}^r {{u_i}f_i^1(} {y^1},{y^2}),\nonumber \\
 &&{{\dot y}^2} = f_0^2({y^2}) + \sum\limits_{i = 1}^r {{u_i}f_i^2(} {y^2}). \nonumber
\end{eqnarray}
The system ${{\dot y}^2} = f_0^2({y^2}) + \sum\limits_{i = 1}^r {{u_i}f_i^2(} {y^2})$ is then a quotient system
of the affine control system (\ref{Equ1}) on the neighborhood $V$. It is said to be induced by
an invariant distribution of the system (\ref{Equ1}). The following proposition links the question (i) with the existence
of an invariant distribution for the system (\ref{Equ6}). We first introduce some notations. Let $e^{tX}(p)$ denote the flow of a vector field $X$
through a vector field $p$ after time t. Let $\mathcal {F}$ be a family of smooth vector fields on
a manifold $M$. We shall use $Lie^\infty \mathcal {F}$  to denote the Lie algebra generated by the vector fields in $\mathcal {F}$. The orbit of $\mathcal {F}$ through a point $p \in M$ is the set of all points ${e^{{t_1}{X_1}}}{e^{{t_2}{X_2}}} \cdots {e^{{t_k}{X_k}}}(p)$ for any vector fields $X_j \in \mathcal {F}$
and numbers $t_j$ for which this is defined.
\begin{proposition}{Proposition}\label{result2}
The question (i) is solvable if and only if there exists a quotient mechanical control system
$(\tilde Q,\tilde \nabla ,\tilde \Im _0 ,d)$ induced by an invariant distribution of the system (\ref{Equ6}).
\end{proposition}
\begin{proof}
To prove the "only if" part, we need to show that the quotient map $\tau$ is a submersion. Then the result follows immediately according to the local representatives for submersions.

According to the condition (\ref{Equ7}), $\tau$ takes the orbit of $\left\{ {S,g_1^{vlft}, \cdots ,g_m^{vlft}} \right\}$
into the orbit of $\left\{ {\tilde S + \tilde g_0^{vlft}  + \tilde d^{vlft}, \tilde g_1^{vlft}, \cdots ,\tilde g_m^{vlft}} \right\}.$
Since the system $(Q,\nabla ,\Im )$ is geodesically accessible,
the orbit of $\left\{ {S,g_1^{vlft}, \cdots ,g_m^{vlft}} \right\}$ equals to $U$ when restricted to the neighborhood $U$ on $TQ$.
 Then we know that the orbit of $\left\{ {\tilde S + \tilde g_0^{vlft}  + \tilde d^{vlft}, \tilde g_1^{vlft}, \cdots ,\tilde g_m^{vlft}} \right\}$ equals to the neighborhood $V$ because $\tau$ is a surjective map. According to the Orbit Theorem  \cite{10}, we know that
\[Lie^\infty \left\{ {\tilde S + \tilde g_0^{vlft}  + \tilde d^{vlft}, \tilde g_1^{vlft}, \cdots ,\tilde g_m^{vlft}} \right\}_{({\tilde x}_0, {\tilde v}_0)}= T_{({\tilde x}_0, {\tilde v}_0)}V, \forall ({\tilde x_0},{\tilde v_0}) \in V.\]
This yields that $\tau$ is a submersion. The proof of the "if" part is obvious.
\end{proof}
Following we will investigate when a geodesically accessible affine connection control system admits
a quotient mechanical control system induced by an invariant distribution.
The notion of geodesical accessibility was
introduced in \cite{8}, from which we have known all structure information about
this subclass of systems is encoded in the Lie algebra generated by Lie
brackets for the system's vector fields. We recall some results
given in \cite{8}.
Given an affine control system ($\ref{Equ1}$) on $M$, define the
following sequence of families of vector fields:

\begin{eqnarray}\label{Equ13}
 &&\nu _1  = \left\{ {f_i :1 \le i \le r} \right\}\nonumber \\
 &&\nu _2  = \left\{ {[f_i ,ad_{f_0} f_j ]:1 \le i,j \le r} \right\} \nonumber \\
 &&\cdots \nonumber \\
 &&\nu _i  = \mathop  \cup \limits_{p + l = i} [\nu _p ,ad_{f_0} \nu _l ].
\end{eqnarray}
Put \[ \nu : = span_R \mathop  \cup \limits_{i = 1}^\infty  \nu _i.
\]
Then we have:
\begin{theorem}{Theorem} \cite{8} \label{result3}
        Let M be a smooth 2n-dimensional manifold. The system ($\ref{Equ1}$) is locally at $y_0  \in M$,
        state equivalent to a geodesically accessible mechanical system  if and only
        if\\
(1)$\dim \nu (y) = n$ and $\dim (\nu  + [f_0,\nu ])(y) = 2n$,\\
(2)$[\nu ,\nu ](y) = 0$,\\
for any $y$ in a neighborhood of $y_0$.\\
The system ($\ref{Equ1}$) is locally at  $y_0  \in M$, state
equivalent to a geodesically accessible mechanical system around a zero-velocity point $(x_0 ,0)$ if and only
if besides the conditions (1) and (2) hold, we have\\
(3) $f_0(y_0 ) \in \nu (y_0 )$.\\
\end{theorem}
The following proposition indicates the uniqueness of the mechanical structure of the quotient
mechanical control system induced by an invariant distribution for the geodesically accessible
mechanical system.

\begin{proposition}{Proposition}\label{result4}
         If a geodesically accessible mechanical system ($\ref{Equ4}$) admits a quotient
         mechanical control system induced by an invariant distribution, then the
         quotient system is geodesically accessible.
\end{proposition}
\begin{proof}
We use the the local coordinates $(U,x{}^i)$ in a neighborhood of
$(q,v) \in TQ$ such that the invariant distribution $\tilde D =
span\left\{ {\frac{\partial }{{\partial x^1 }}, \cdots
,\frac{\partial }{{\partial x^{2k} }}} \right\}$.\\
Denote $x = (x^1 , \cdots ,x^{2k} ),y = (x{}^{2k + 1}, \cdots
,x^{2n})$. Then ($\ref{Equ4}$) takes the form
\begin{eqnarray}\label{Equ14}
 &&\dot x = f^1 (x,y) + \sum\limits_{i = 1}^m {u_i } g_i^1 (x,y),\nonumber \\
 &&\dot y = f^2 (y) + \sum\limits_{i = 1}^m {u_i } g_i^2 (y),
\end{eqnarray}
where
\begin{equation}\label{Equ15}
\dot y = f^2 (y) + \sum\limits_{i = 1}^m {u_i } g_i^2 (y)
\end{equation}
is a mechanical system.\\
According to the construction ($\ref{Equ13}$) of Lie algebra $\nu $
for the system ($\ref{Equ4}$), we conclude that any vector field $G \in \nu $
has the form
\begin{equation}\label{Equ16}
G = \left( {\begin{array}{*{20}c}
   {G_1 }  \\
   {G_2 }  \\
\end{array}} \right),
\end{equation}
where $G_2 $ belongs to the Lie algebra  $\tilde \nu $ given
by ($\ref{Equ13}$) for the system ($\ref{Equ15}$). Since the system ($\ref{Equ4}$) is
geodesically accessible, we have $\dim (\nu  + [f,\nu ])(z) = 2n$
where $z = (z^1 ,z^2 )$ is in a neighborhood of $(q,v) = (z_0^1
,z_0^2 )$.\\
 This gives
\begin{equation}\label{Equ17}
\dim (\tilde \nu  + [f^2 ,\tilde \nu ])(z^2 ) = 2n - 2k,
\end{equation}
for $z^2 $ in a neighborhood of $z_0^2 $. \\So
\begin{equation}\label{Equ18}
\dim \tilde \nu (z^2 ) \ge n - k.
\end{equation}
On the other hand, since ($\ref{Equ15}$) is a mechanical control
system with dimension $2n-2k$, we have
\begin{equation}\label{Equ19}
\dim \tilde \nu (z^2 ) \le n - k.
\end{equation}
Then we derive
\begin{equation}\label{Equ20}
\dim \tilde \nu (z^2 ) = n - k,
\end{equation}
for $z^2 $  in a neighborhood of $z_0^2 $ according to ($\ref{Equ18}$) and
($\ref{Equ19}$).\\
From ($\ref{Equ16}$), we know that $\left[ {\nu ,\nu } \right](z) = 0$ yields $\left[ {\tilde \nu ,\tilde \nu } \right]({z^2}) = 0$.\\
So the mechanical system ($\ref{Equ15}$) is geodesically accessible according to Theorem $\ref{result3}$.
\end{proof}

\begin{remark}{Remark}
If locally around a zero velocity point, the geodesically accessible mechanical system ($\ref{Equ4}$) admits a quotient
mechanical control system induced by an invariant distribution , then from ($\ref{Equ14}$) and ($\ref{Equ16}$). we know
the quotient mechanical system ($\ref{Equ15}$) is accessible. Besides, it is also around a zero velocity point according to the condition (3) in Theorem $\ref{result3}$. Then we derive that the system ($\ref{Equ15}$) is also geodesically accessible because accessibility at the point of zero velocity is equivalent to geodesical accessibility.
However, if the geodesically accessible mechanical system ($\ref{Equ4}$) is not around a zero velocity point, we can not assume that the system ($\ref{Equ15}$)
is  around a zero velocity point. So the proof of the above proposition is
necessary since the equivalence between accessibility at the point with nonzero velocity and  geodesical accessibility may not hold.
\end{remark}

For fully actuated mechanical systems, we have:
\begin{corollary}{Corollary}
The quotient mechanical system induced by an invariant
distribution for a fully actuated mechanical control system is fully actuated.
\end{corollary}
\begin{proof}
It is obvious from the proof of Proposition $\ref{result4}$.
\end{proof}
Following we will focus on recognizing the structure of the invariant distribution that gives a quotient mechanical control system.
\begin{proposition}{Proposition}\label{result5}
        If the geodesically accessible mechanical control system ($\ref{Equ4}$)
        admits a quotient mechanical control system induced by a 2k-dimensional invariant distribution $\tilde D$,
        then there exists a k-dimensional distribution $D$ on $Q$ such that ${D^{vlft}}$ is contained in $\tilde D$.
\end{proposition}
\begin{proof}
 We use the local
coordinates $(U,x{}^i)$ in a neighborhood of $(q,v) \in TQ$ such
that the invariant distribution $ \tilde D = span\left\{
{\frac{\partial }{{\partial x^1 }}, \cdots ,\frac{\partial
}{{\partial x^{2k} }}} \right\}$.\\ Denote $x = (x^1 , \cdots
,x^{2k} ),y = (x{}^{2k + 1}, \cdots ,x^{2n} )$. Then the system ($\ref{Equ4}$)
takes the form
\begin{eqnarray}
 &&\dot x = f^1 (x,y) + \sum\limits_{i = 1}^m {u_i } g_i^1 (x,y),\nonumber \\
 &&\dot y = f^2 (y) + \sum\limits_{i = 1}^m {u_i } g_i^2 (y). \nonumber
\end{eqnarray}
From Proposition $\ref{result4}$, we know that
\begin{equation}\label{Equ21}
\dot y = f^2 (y) + \sum\limits_{i = 1}^m {u_i } g_i^2 (y)
\end{equation}
is a geodesically accessible mechanical system. \\
Choose n vector
fields  $V_1 , \cdots ,V_n$
 in the Lie algebra $\nu $ for the system ($\ref{Equ4}$) which are linearly independent at any $z$ in a neighborhood of
$(q,v) = (z_0^1 ,z_0^2 )$. \\
Locally,
\[V_i  = (V_i^1, V_i^2 ),\]
where $V_i^2 $ is in the Lie algebra $\tilde \nu $ for the system
($\ref{Equ21}$) and \[span\left\{ {V_i^2 :i = 1, \cdots ,n}
\right\}(z^2 ) = \tilde \nu
(z^2 ),\] for $z^2$ in a neighborhood of $z_0^2 $.\\
Since $\dim \tilde \nu (z^2 ) = n - k$, assume $V_1^2 , \cdots ,V_{n
- k}^2 $ are linearly independent, then we obtain there exist
functions $\alpha _i^j  \in C^\infty  (y), i = n - k + 1, \cdots, n, j
= 1, \cdots, n - k$ such that
\begin{equation}\label{Equ22}
V_i^2  = \alpha _i^j V_j^2 ,i = n - k + 1, \cdots, n, j = 1, \cdots, n - k.
\end{equation}
Let \[\tilde V_i  = V_i  - \alpha _i^j V_j, i = n - k + 1, \cdots, n, j = 1, \cdots ,n - k.\]
Then the $k$  linearly independent vector fields $\tilde V_{n - k +
1,} \cdots, \tilde V_n \in \Gamma^\infty (\tilde D)$.\\
From ($\ref{Equ22}$) we know \[[\alpha _i^j V_j^2 ,V_l^2 ] = 0, i = n
- k + 1, \cdots, n, j = 1, \cdots, n - k, l = 1, \cdots, n,\] which
yields \[V_l^2 (\alpha _i^j ) = 0, i = n - k + 1, \cdots, n, j = 1,
\cdots, n - k, l= 1, \cdots, n.\]
Since $\alpha _i^j  \in C^\infty  (y)$, we have
\begin{equation}\label{Equ23}
V_l (\alpha _i^j ) = 0, i = n - k + 1, \cdots, n, j = 1, \cdots, n -
k, l= 1, \cdots ,n.
\end{equation}
Consider the local coordinates $(V,z{}^i)$ in a neighborhood of $q
\in Q$ which give local coordinates $(z{}^1, \cdots ,z^n ,y^1 ,\cdots ,y^n )$
 on $TQ$. In such coordinates, \[span\left\{ {V_1 , \cdots ,V_n } \right\} =
span\left\{ {\frac{\partial }{{\partial y^1 }}, \cdots
,\frac{\partial }{{\partial y^n }}} \right\}.\]\\
So $\alpha _i^j  \in C^\infty  (z^1 , \cdots ,z^n )$ according to
($\ref{Equ23}$).\\
Then $\tilde V_i$ has the form  $\tilde V_i  = v_i ^{vlft}$
where $v_i$ is a vector field on $Q$.
This completes the proof for the proposition.
\end{proof}
Now we start to prove Theorem \ref{result1}.
\begin{proof}
We need to prove the following proposition:

The geodesically accessible affine connection control system ($\ref{Equ6}$) admits a quotient  mechanical system induced by an invariant distribution if and only if there exists an involutive distribution $D$  on $Q$ with the affine connection $\nabla$ restricting to $D$, and the curvature tensor satisfying $R(X,v)v \in \Gamma^\infty  (D)$, for any $X \in \Gamma^\infty (D)$, $v \in \Gamma^\infty (TQ)$, besides,  $[g_i ,D] \subseteq D,i = 1, \cdots ,m$.\\ Then Theorem \ref{result1} follows immediately according to Proposition \ref{result2}.

To prove the "only if " part, let $\tilde D$ be a 2k-dimensional invariant distribution for the geodesically accessible
affine connection control system ($\ref{Equ6}$),
which gives a quotient mechanical control system. Then according to Proposition \ref{result5}, we have
\[{D^{vlft}} \subseteq \tilde D,\]
where $D$ is a k-dimensional distribution on $Q$.
Since $\tilde D$ is invariant under the geodesic spray $S$, we have
\[\left[ {S,{D^{vlft}}} \right]  \subseteq \tilde D.\]
For $\dim span\left\{ {{D^{vlft}},\left[ {S,{D^{vlft}}} \right]} \right\} = 2k = \dim \tilde D$, this yields
 \[\tilde D = span\left\{ {{D^{vlft}},[S,{D^{vlft}}]} \right\}.\]
Since  $[S,[S,{D^{vlft}}]] \subseteq \tilde D$, then we know
\begin{equation}\label{Equ24}
{\nabla _{v}}X \in \Gamma^\infty (D), \forall X \in \Gamma^\infty (D), v\in \Gamma^\infty (TQ),
\end{equation}
according to ($\ref{Equ10}$) and ($\ref{Equ12}$).
That is, the affine connection $\nabla$ restricts to $D$.
This gives
\begin{equation}\label{Equ25}
\tilde D = span\left\{ {{D^H},{D^{vlft}}} \right\}.
\end{equation}
Then we derive \[(R(X, v){v} + ver(\nabla _{v^H}^H{({\nabla _{v}}X)^{vlft}}))\in \Gamma^ \infty (D),\]
for any $X \in \Gamma^ \infty (D)$, $v\in \Gamma^ \infty (TQ)$ according to ($\ref{Equ12}$).\\
Since $\nabla$ restricts to $D$, we have $ver(\nabla _{v^H}^H{({\nabla _{v}}X)^{vlft}})\in \Gamma^ \infty (D)$ for any $X \in \Gamma^ \infty (D)$ according to the definition
of ${\nabla ^H}$ (see \cite{11}).\\
So
\begin{equation}\label{Equ26}
R(X,v)v \in \Gamma^\infty(D), \forall X \in \Gamma^\infty (D),  v \in \Gamma^\infty (TQ).
\end{equation}
For $\tilde D$ is involutive, we have
\[\left[ {\left[ {S,{D^{vlft}}} \right],\left[ {S,{D^{vlft}}} \right]} \right] \subseteq \tilde D.\]
Then ($\ref{Equ10}$) gives that $D$ is involutive.\\
Since $\left[ {g_i^{vlft},\tilde D} \right] \subseteq \tilde D,i = 1, \cdots ,m,$ according to ($\ref{Equ25}$),
we have \[\left[ {g_i^{vlft},{X^H}} \right] = {({\nabla _X}{g_i})^{vlft}} \in \Gamma^\infty (\tilde D),\forall X \in \Gamma^\infty (D), ,i = 1, \cdots ,m.\]
This yields
${\nabla _X}{g_i} \in \Gamma^\infty (D),\forall X \in \Gamma^\infty (D),i = 1, \cdots ,m.$\\
Then we have $\left[ {X,{g_i}} \right] = {\nabla _X}{g_i} - {\nabla _{{g_i}}}X \in \Gamma^\infty (D),\forall X \in \Gamma^\infty (D),i = 1, \cdots ,m.$
That is,
\begin{equation}\label{Equ27}
\left[ {{g_i},D} \right] \subseteq D,i = 1, \cdots ,m.
\end{equation}
In fact, in local coordinates $(x,y)$ such that $D = span\left\{ {\frac{\partial }{{\partial {x^1}}}, \cdots ,\frac{\partial }{{\partial {x^k}}}} \right\}$, by computing we have
\[\tilde D = span\left\{ {\frac{\partial }{{\partial {x^1}}}, \cdots ,\frac{\partial }{{\partial {x^k}}},\frac{\partial }{{\partial {y^1}}}, \cdots ,\frac{\partial }{{\partial {y^k}}}} \right\}.\]
Then under the local coordinates $(x,y)$, the quotient system has the form
\begin{eqnarray}\label{Equ28}
 &&{{\dot x}^{k + 1}} = {y^{k + 1}}, \nonumber \\
 && \cdots  \nonumber \\
 &&{{\dot x}^n} = {y^n}, \nonumber \\
 &&{{\dot y}^{k + 1}} =  - \sum\limits_{k + 1 \le i,j \le n} {\Gamma _{ij}^{k + 1}({x^{k + 1}}, \cdots ,{x^n}){y^i}{y^j}}  + \sum\limits_{i = 1}^m {{u^i}} g_i^{k + 1}({x^{k + 1}}, \cdots ,{x^n}), \nonumber \\
  &&\cdots  \nonumber \\
 &&{{\dot y}^n} =  - \sum\limits_{k + 1 \le i,j \le n} {\Gamma _{ij}^n({x^{k + 1}}, \cdots ,{x^n}){y^i}{y^j}}  + \sum\limits_{i = 1}^m {{u^i}} g_i^n({x^{k + 1}}, \cdots ,{x^n}),
\end{eqnarray}
 which is an affine connection control system.

To prove the "if" part, suppose there exists an involutive k-dimensional distribution $D$ satisfying (\ref{Equ24}),(\ref{Equ26}) and (\ref{Equ27}).
Let \[\tilde D = span\left\{ {{D^{vlft}},[S,{D^{vlft}}]} \right\}.\]
Since the affine connection $\nabla$ restricts to $D$, we know equation (\ref{Equ25}) holds.
According to (\ref{Equ12}), (\ref{Equ24}) and (\ref{Equ26}), $\left[ {S,\left[ {S,{D^{vlft}}} \right]} \right] \subseteq span\left\{ {{D^H},{D^{vlft}}} \right\}.$
Then
\begin{equation}\label{Equ29}
\left[ {S,\tilde D} \right] \subseteq span\left\{ {\tilde D,\left[ {S,\left[ {S,{D^{vlft}}} \right]} \right]} \right\} = \tilde D.
\end{equation}
For \[\left[ {g_i^{vlft},{X^H}} \right] = {({\nabla _X}{g_i})^{vlft}} = {({\nabla _{{g_i}}}X + \left[ {X,{g_i}} \right])^{vlft}} \in \Gamma^\infty ({D^{vlft}}),\forall X \in \Gamma^\infty (D),\]
this yields
\begin{equation}\label{Equ30}
\left[ {g_i^{vlft},\tilde D} \right] \subseteq \tilde D,i = 1, \cdots ,m.
\end{equation}
Since $D$ is involutive and the affine connection $\nabla$  restricts to $D$, it is obvious that $\tilde D$ is involutive by computing.
Then it follows from ($\ref{Equ29}$) and ($\ref{Equ30}$) that $\tilde D$ is an invariant distribution for the geodesically accessible
affine connection control system (\ref{Equ6}). The left thing is to show that the quotient system induced by $\tilde D$ is mechanical.   In local coordinates $(x,y)$ such that $D = span\left\{ {\frac{\partial }{{\partial {x^1}}}, \cdots ,\frac{\partial }{{\partial {x^k}}}} \right\}$, by computing we have
\[\tilde D = span\left\{ {\frac{\partial }{{\partial {x^1}}}, \cdots ,\frac{\partial }{{\partial {x^k}}},\frac{\partial }{{\partial {y^1}}}, \cdots ,\frac{\partial }{{\partial {y^k}}}} \right\}.\]
Then under the local coordinates $(x,y)$, the quotient system has the form (\ref{Equ28}) which is mechanical.
\end{proof}
In the proof of Theorem \ref{result1}, together with Proposition \ref{result4}, we have also derived the following result, which shows that
the quotient mechanical system induced by an invariant distribution for a
geodesically accessible affine connection control system
must inherit its mechanical structure from the
original system.
\begin{corollary}{Corollary}\label{result6}
If question (i) is solvable,
then the quotient map $\tau :U \to V$ must has the form $\tau  = T\Phi$ with $\Phi : U_1 \to V_1 $ a submersion,
where $ U_1 $ and $V_1 $ are neighborhoods of $x_0 \in Q$ and ${\tilde x}_0 \in \tilde Q$.
Besides, the quotient mechanical control system is also a geodesically accessible affine connection control system.
\end{corollary}
For affine connection control systems that are not geodesically accessible, the above corollary may not
hold. The quotient mechanical control system may even not be an affine
connection system. We give an example to explain this.
\begin{example}{Example}
 Consider the following affine connection control system on ${\mathbb{R}}^3$:
\begin{eqnarray}\label{Equ31}
 &&\dot x^1  = y^1, \nonumber \\
 &&\dot x^2  = y^2, \nonumber\\
 && \dot x^3  = y^3,  \nonumber\\
 && \dot y^1  = (y^1)^2  + y^1 y^2, \nonumber\\
 && \dot y^2  = (y^1)^2  - (y^2)^2  + y^1 y^2  + u, \nonumber\\
 &&\dot y^3  = v.
\end{eqnarray}
Given a submersion $\tau :{\mathbb{R}}^6  \to {\mathbb{R}}^2$ with
$\tau (x^1 ,x^2 ,x^3 ,y^1 ,y^2 ,y^3 ) = (y^1 ,y^2 )$, we get the
quotient system
\begin{eqnarray}\label{Equ32}
 &&\dot y^1  = (y^1)^2  + y^1 y^2,  \nonumber\\
 &&\dot y^2  = (y^1)^2  - (y^2)^2  + y^1 y^2  + u.
\end{eqnarray}
Let \begin{eqnarray}
 &&\tilde x^1  = y^1,  \nonumber\\
 &&\tilde x^2  = (y^1)^2  + y^1 y^2, \nonumber
\end{eqnarray}
 with $y^1  < 0$, then the system ($\ref{Equ32}$) becomes
\begin{eqnarray}\label{Equ33}
 &&{\dot{\tilde x}}^1  = \tilde x^2,  \nonumber\\
 &&{\dot{\tilde x}}^2  = 4{\tilde x}^1 {\tilde x}^2  - (\tilde x^1)^3  + {\tilde x}^1 u.
\end{eqnarray}
The system ($\ref{Equ33}$) is a mechanical control system on $\tilde Q = \left\{
{x:x < 0} \right\}$. It is not an affine connection system. And it doesn't inherit its mechanical structure from the system
($\ref{Equ31}$).
\end{example}
\section{Global properties}\label{sec4}
In this section, we study question (ii) given in the introduction. For the affine connection system (\ref{Equ5}), we use $\Im$ to denote the smallest sets of vector fields containing $g_i, i=1, \cdots,m$ and such that it is closed under the symmetric product. We make an assumption that the vector fields in $\Im$ are complete.
\begin{theorem}{Theorem} \cite{12} \label{result7}
If $\Im _j$ are sets of vector fields on manifold $M_j$, for $j=0,1$, and $\phi: M_0 \to M_1$ satisfies $\phi _ *\Im _0=\Im _1$, if the vector fields in both families are complete, then $\phi$ is a fiber bundle map on each orbit.
\end{theorem}

Recall that the vertical distribution for a fiber bundle $E$ is given by the vertical subspace of $TE$ (see \cite{9}), we have
\begin{theorem}{Theorem}\label{result8}
Question (ii) is solvable if and only if the configuration manifold $Q$ admits a fiber bundle structure such that the vertical distribution $D$ satisfying that
the affine connection $\nabla$ restricting to $D$, and the curvature tensor satisfying $R(X,v)v \in \Gamma^\infty  (D)$, for any $X \in \Gamma^\infty (D)$, $v \in \Gamma^\infty (TQ)$, besides, $[g_i ,D] \subseteq D,i = 1, \cdots ,m$.
\end{theorem}
\begin{proof}
If question (ii) is solvable, according to (\ref{Equ7}), we have
\begin{equation}\label{Equ34}
{\tau _*}{\nu _i} = {\tilde \nu _i},i \in {{\rm N}_{ > 0}},
\end{equation}
where $\nu_i$ and $\tilde \nu_i$ are the sets of vector fields given in (\ref{Equ13}), respectively  for the geodesically accessible affine connection control system $(Q,\nabla ,\Im )$ and its quotient system $(\tilde Q,\tilde \nabla ,\tilde \Im _0 ,d).$

On the other hand, similar with the proof of Proposition \ref{result2}, we know the quotient map $\tau : TQ \to T\tilde Q$  is a submersion.
Then according to Section \ref{sec3}, $\tau$  has the form $\tau  = T\Phi$ with $\Phi : Q \to \tilde Q$  being a submersion, further, the quotient mechanical control system is also a geodesically accessible affine connection control system. Thus according to (\ref{Equ34}),
we have
\[\Phi_*\Im=\tilde \Im,\]
where $\tilde \Im$  denotes the smallest sets of vector fields containing $\tilde{g}_i, i=1, \cdots,m$ and such that it is closed under the symmetric product.

Since the orbits of $\Im$ and $\tilde \Im$ are $Q$ and $\tilde Q$ respectively, then from Theorem \ref{result7},  we know that $\Phi : Q \to \tilde Q$  is a fiber bundle. The remaining results hold according to Theorem \ref{result1}. This completes the proof for the "only if" part. The proof of the "if" part is obvious.
\end{proof}
\section{Conclusions}\label{sec5}
In this paper we have studied the quotients of affine connection
control systems. We are interested in the special structures in
the class of affine connection control systems, such that further
results can be derived for the quotient systems. The main part dealt
with questions (i) and (ii) given in the introduction. We gave
necessary and sufficient conditions for geodesically accessible
affine connection control systems such that the questions are solvable.
We showed that all the quotient mechanical control systems induced by
invariant distributions must inherit the mechanical structures from
the original geodesically accessible affine connection systems. For affine connection control
systems that are not geodesically accessible, an example was given
to show the non-uniqueness of the form of the quotient map.



\end{document}